\documentclass[11pt]{amsart}

\usepackage{amssymb,amsmath,amsthm,verbatim}
\usepackage{fullpage}
\usepackage{xcolor}
\usepackage{url}
\usepackage{mathrsfs}
\usepackage[all]{xy}
  \SelectTips{cm}{10}
  \everyxy={<2.5em,0em>:}
\usepackage{tikz}

\usepackage{fancyhdr}
\usepackage{ifpdf}
  \ifpdf
  \else
    
    \newcommand{\href}[2]{#2}
  \fi

\theoremstyle{plain}
  \newtheorem{lemma}[equation]{Lemma}
  \newtheorem{proposition}[equation]{Proposition}
  \newtheorem{theorem}[equation]{Theorem}
  \newtheorem{corollary}[equation]{Corollary}

    \newtheorem{conjecture}[equation]{Conjecture}

\theoremstyle{definition}

\theoremstyle{remark}
  \newtheorem{remark}[equation]{Remark}

\renewcommand{\thesection}{\arabic{section}}
\renewcommand{\theequation}{\thesection.\arabic{equation}}

 \DeclareFontFamily{U}{manual}{}
 \DeclareFontShape{U}{manual}{m}{n}{ <->  manfnt }{}
 \newcommand{\manfntsymbol}[1]{%
    {\fontencoding{U}\fontfamily{manual}\selectfont\symbol{#1}}}

\makeatletter
   \@addtoreset{section}{part}
   \@addtoreset{equation}{section}
   \@addtoreset{footnote}{section}

    {\hspace*{\fill}$\lrcorner$\endgraf\endgroup\end{trivlist}}
 
\makeatother

  \DeclareFontFamily{OT1}{pzc}{}
  \DeclareFontShape{OT1}{pzc}{m}{it}{<-> s * [1.100] pzcmi7t}{}
  \DeclareMathAlphabet{\mathpzc}{OT1}{pzc}{m}{it}

\newif\ifhascomments \hascommentstrue
\ifhascomments
  \newcommand{\jason}[1]{{\color{red}[[\ensuremath{\bigstar\bigstar\bigstar} #1]]}}
  \newcommand{\matt}[1]{{\color{red}[[\ensuremath{\spadesuit\spadesuit\spadesuit} #1]]}}
\else
  \newcommand{\jason}[1]{}
  \newcommand{\matt}[1]{}
\fi

\DeclareMathOperator{\aut}{Aut}
\DeclareMathOperator{\Aut}{\ensuremath{\mathcal{A}\kern-.125em\mathpzc{ut}}}

\newcommand{\bbar}[1]{\overline{#1}}

\newcommand{\CC}{\mathbb C}

\DeclareMathOperator{\Endo}{\ensuremath{\mathcal{E}\kern-.125em\mathpzc{nd}}}

\newcommand{\GG}{\mathbb G}
\DeclareMathOperator{\GL}{GL}

\DeclareMathOperator{\Hom}{\ensuremath{\mathcal{H}\kern-.125em\mathpzc{om}}}
\newcommand{\id}{\mathrm{id}}

\renewcommand{\L}{\mathcal L}

\newcommand{\M}{\mathcal M}
\newcommand{\N}{\mathcal N}

\newcommand{\NN}{\mathbb N}
\renewcommand{\O}{\mathcal O}

\newcommand{\PP}{\mathbb{P}}

\newcommand{\QQ}{\mathbb Q}
\newcommand{\QQbar}{{\bbar{\mathbb Q}}}
\newcommand{\RR}{\mathbb R}

\renewcommand{\setminus}{\smallsetminus}

\DeclareMathOperator{\spec}{Spec}

\newcommand{\ZZ}{\mathbb{Z}}

 \def\ari[#1]{\ar@{^(->)}[#1]}
 \def\are[#1]{\ar[#1]^{\txt{\'et}}}
 \def\areh[#1]{\ar[#1]|{\txt{$H$-eq}}^{\txt{\'et}}}
 \def\ars[#1]{\ar@{->>}[#1]}
 \newcommand{\dplus}{\ar@{}[d]|{\mbox{$\oplus$}}}
 \newcommand{\dtimes}{\ar@{}[d]|{\mbox{$\times$}}}

\usepackage[all]{xy}

\xyoption{all}

\DeclareMathOperator{\OO}{{\mathcal{O}}}
\DeclareMathOperator{\Qbar}{{\overline{\mathbb{Q}}}}

\DeclareMathOperator{\Amp}{Amp}
\DeclareMathOperator{\bir}{Bir}

\DeclareMathOperator{\PGL}{PGL}

\DeclareMathOperator{\Nef}{Nef}
\DeclareMathOperator{\NS}{NS}
\DeclareMathOperator{\Pic}{Pic}

\newcommand{\fp}{{\mathfrak p}}

\definecolor{darkgreen}{rgb}{0.0, 0.2, 0.13}
\definecolor{darkpastelgreen}{rgb}{0.01, 0.75, 0.24}

\begin{document}
\title{On the Medvedev-Scanlon Conjecture for Minimal Threefolds of non-negative Kodaira Dimension}

\author{J.~P.~Bell}
\address{Jason P. Bell\\
Department of Pure Mathematics\\
University of Waterloo\\
Waterloo, ON N2L 3G1\\
CANADA 
}
\email{jpbell@uwaterloo.ca}

\author{D.~Ghioca}
\address{
Dragos Ghioca\\
Department of Mathematics\\
University of British Columbia\\
Vancouver, BC V6T 1Z2\\
Canada
}
\email{dghioca@math.ubc.ca}

\author{Z.~Reichstein}
\address{
Zinovy Reichstein\\
Department of Mathematics\\
University of British Columbia\\
Vancouver, BC V6T 1Z2\\
Canada
}
\email{reichst@math.ubc.ca}

\author{M.~Satriano}
\address{
Matthew Satriano\\
Department of Pure Mathematics\\
University of Waterloo\\
Waterloo, ON N2L 3G1\\
CANADA 
}
\email{msatriano@uwaterloo.ca}

\thanks
{The authors have been partially supported by Discovery Grants from
the National Science and Engineering Board of Canada.}

\begin{abstract} Motivated by work of Zhang from the early `90s, 
Medvedev and Scanlon formulated the following conjecture. 
Let $F$ be an algebraically closed field of characteristic $0$  
and let $X$ be a quasiprojective variety defined over $F$ 
endowed with a dominant rational self-map $\Phi$. Then there 
exists a point $x\in X(F)$ with Zariski dense orbit 
under $\Phi$ if and only if $\Phi$ preserves no nontrivial 
rational fibration, i.e.,~there exists no non-constant rational functions
$f\in F(X)$ such that $\Phi^*(f)=f$. The Medvedev-Scanlon conjecture holds when $F$ is uncountable. The case where 
$F$ is countable (e.g., $F=\Qbar$) is much more difficult; here
the Medvedev-Scanlon conjecture has only been proved in a small number 
of special cases. In this paper we show that
the Medvedev-Scanlon conjecture holds for all varieties 
of positive Kodaira dimension, and explore the case 
of Kodaira dimension $0$. 
Our results are most complete in dimension $3$. 
\end{abstract}



\maketitle

\tableofcontents

\section{Introduction}
\label{sec:intro}

Consider a dominant rational self-map $\phi \colon X\dashrightarrow X$ of 
an irreducible variety $X$, defined over a field $k$.
For an integer $n\ge 0$, we will denote by $\phi^n$ 
the $n$-th compositional power of $\phi$. Given a point $x \in X$, we define its orbit under $\phi$ (denoted $\OO_\phi (x)$) to be 
the set of all $\phi^n(x)$ (as $n$ ranges over the non-negative integers) whenever $x$ is not in the indeterminacy locus for $\phi^n$. 

In this paper, we will prove the Medvedev-Scanlon conjecture 
for a large class of projective varieties $X$.  This is a conjecture in arithmetic dynamics that predicts when there is 
a point in $X(\bbar\QQ)$ with dense $\phi$-orbit. 
Certainly, no such $\bbar\QQ$-point can exist 
if $\phi$ \emph{preserves a rational fibration}, 
i.e.~if there is a dominant rational map 
$\pi:X\dasharrow Y$ with $\dim Y>0$ such that $\pi\circ\phi=\pi$. 
The Medvedev-Scanlon conjecture asserts that this necessary condition
is also sufficient.

\begin{conjecture}[{\cite[7.14]{MS}}]
\label{conj:MS}
Let $X$ be an irreducible variety over an algebraically 
closed field $F$ of characteristic 0 
and $\phi \colon X \dasharrow X$ be a dominant rational self-map. 
If $\phi$ does not preserve a rational fibration, then there is 
a point $x\in X(F)$ with Zariski dense forward orbit under $\phi$.
\end{conjecture}

In the case, where $F$ is uncountable, Conjecture~\ref{conj:MS}
was proved earlier by Amerik and Campana~\cite[Theorem~4.1]{A-C} 
(and under the stronger hypothesis that $\phi$ is an automorphism 
of $X$ independently by Bell, Rogalski 
and Sierra~\cite[Theorem~1.2]{Dixmier}). 
Conjecture~\ref{conj:MS} was, in fact, motivated by this theorem and 
by an older conjecture of Zhang \cite[Conjecture~4.1.6]{Zhang} 
about Zariski dense orbits for polarizable endomorphisms.

For the rest of the introduction we will assume that $F$ is a countable 
algebraically closed field of characteristic $0$ (e.g., $F = \bar{\QQ}$).
Here the Medvedev-Scanlon conjecture has only been proved in 
a few special cases, using subtle diophantine techniques:

\smallskip
(1) Medvedev and Scanlon \cite[Theorem~7.16]{MS} established
Conjecture~\ref{conj:MS}  for endomorphisms $\phi$ of $X=\mathbb{A}^m$ 
of the form $\phi(x_1,\dots, x_m)=(f_1(x_1),\dots, f_m(x_m))$, 
where $f_1,\dots, f_m\in F[x]$. 
Their proof combines techniques from model theory, 
number theory and polynomial decomposition theory 
to obtain a complete description of all 
periodic subvarieties. 

\smallskip
(2) In the case where $X$ is an abelian variety and $\phi \colon X \to X$ 
is dominant self-map,  Conjecture~\ref{conj:MS} was proved 
by Ghioca and Scanlon~\cite{MS-ab-var}. The proof uses an explicit description
of endomorphisms of an abelian variety and relies on 
the Mordell-Lang conjecture, due to  Faltings~\cite{Faltings}. 

\smallskip
(3) In the case where $\dim(X) \leq 2$ and $\phi \colon X \dasharrow X$
is a birational isomorphism, Conjecture~\ref{conj:MS} was established 
by Xie~\cite{xie}. We remark that in~\cite[Theorem~1.4]{xie}, 
this result is stated under the additional assumption that
the first dynamical degree of $\phi$ is greater than $1$; 
however, the same proof goes through without this assumption.
We will not use~\cite[Theorem~1.4]{xie} in this paper, but we will
appeal to the case of regular automorphisms of surfaces, which was 
settled earlier in~\cite[Theorem~1.3]{BGT-2}. 
These results  are proved by $p$-adic 
techniques, in particular, the so-called $p$-adic arc lemma. 
For details on the $p$-adic arc lemma and its applications 
we refer the reader to \cite[Chapter~4]{book}.  


\smallskip
(4) Xie \cite[Theorem~1.1]{xie-2} recently proved 
Conjecture~\ref{conj:MS} for all polynomial endomorphisms 
of $\mathbb{A}^2$. The proof relies on valuation-theoretic techniques.

\smallskip
In this paper we will explore Conjecture~\ref{conj:MS} in
the case where $\phi \colon X \dasharrow X$ is a birational automorphism
and $\dim(X) \geq 3$ by 
using techniques of higher-dimensional algebraic geometry. 
Our first main result settles the Medvedev-Scanlon conjecture 
for birational self-maps of varieties of positive Kodaira dimension.

\begin{theorem}
\label{thm:MS-Kod>0}
If $X$ is an irreducible projective variety of Kodaira dimension $\kappa(X)>0$ 
defined over a field of characteristic $0$ and $\phi \colon X \dasharrow X$ is a birational self-map, 
then $\phi$ preserves a rational fibration.
In particular, the Medvedev-Scanlon Conjecture \ref{conj:MS} is 
vacuously true in this case.
\end{theorem}

Our next result shows that if $X$ is a smooth minimal 
model with $\kappa(X) = 0$, then assuming standard conjectures in 
the minimal model program, Conjecture~\ref{conj:MS} can be reduced 
to products of special kinds of varieties: Calabi-Yau, hyperk\"ahler, 
and abelian varieties. 

Recall that a smooth projective variety $X$ over $\QQbar$ 
is called \emph{hyperk\"ahler} 
if its complex analytification is simply connected 
and $H^0(\Omega^2_X)$ is spanned by a symplectic form. 
In dimension 2, hyperk\"ahler varieties are nothing more than K3 surfaces.

We use the convention that a smooth projective variety 
of dimension $\geq 3$ defined over $\QQbar$ is \emph{Calabi-Yau} 
if the complex analytification $X_{\CC}$ is simply connected, 
$K_X\simeq\O_X$, and $H^p(\O_X)=0$ for $0<p<\dim X$. 
Since we are working over $\QQbar$, by the symmetry of the Hodge 
diamond, this latter condition is equivalent 
to requiring $H^0(\Omega^p_X)=0$ for $0<p<\dim X$.

Our next result relies on the abundance conjecture, which is known for curves, surfaces, and threefolds. We state it in the form that we need, although the conjecture itself is more general.

\begin{conjecture}[{Abundance \cite[Corollary 3.12]{kollar-mori}}]
\label{conj:abundance}
If $X$ is a smooth projective minimal variety of Kodaira dimension 0, then $K_X$ is numerically trivial.
\end{conjecture}

\begin{remark}
For readers more familiar with diophantine geometry and less familiar with the techniques of the minimal model program, we emphasize that assuming the abundance conjecture is akin to assuming the generalized Riemann hypothesis. Although the result is not yet known in higher dimension, it is largely expected that the conjecture is true.
\end{remark}

\begin{theorem}
\label{thm:MS-reduction-nfolds}
Fix an integer $n \geq 1$.  Assuming the Abundance Conjecture 
\ref{conj:abundance}, the Medvedev-Scanlon 
Conjecture \ref{conj:MS} holds for birational self-maps of smooth projective minimal 
$n$-folds over $\bbar\QQ$ of Kodaira dimension 0 if and only if 
it holds for those $n$-folds of the form 
$A\times \prod_i Y_i\times \prod_j Z_j$, 
where $A$ is an abelian variety, the $Y_i$ are Calabi-Yau, 
and the $Z_j$ are hyperk\"ahler.
\end{theorem}

In the case of threefolds, we unconditionally reduce to the case 
of Calabi-Yau varieties. 

\begin{theorem}
\label{thm:MS-reduction-3folds}
The Medvedev-Scanlon Conjecture \ref{conj:MS} holds for birational self-maps of smooth projective minimal threefolds over $\bbar\QQ$ of Kodaira dimension $0$ if and only if it holds for smooth Calabi-Yau threefolds.
\end{theorem}

Finally, we handle the case of Calabi-Yau threefolds, contingent 
on conjectures in the minimal model program. Via the intersection 
product, the second Chern class $c_2(X)$ defines a linear form on 
the nef cone $\Nef(X)$. Miyaoka \cite{Miyaoka} 
shows that this linear form always assumes non-negative values 
on the nef cone. We separately consider the cases where $c_2(X)$ 
is strictly positive and where it is not.

\begin{theorem}
\label{thm:MS-CY3}
Let $X$ be a smooth projective Calabi-Yau threefold over $\bbar\QQ$. Then the Medvedev-Scanlon Conjecture \ref{conj:MS} holds for all (regular) 
automorphisms $\phi \colon X \to X$ if either:
\begin{enumerate}
\item \label{item:c2>0_semiample} $c_2(X)$ is positive on $\Nef(X)$, or
\item \label{item:c2=0_semiample} there is a semi-ample divisor 
$D\neq0$ on $X$ such that $c_2(X)\cdot D=0$.
\end{enumerate}
\end{theorem}

Here by ``divisor" we mean that $D$ is an integral point of $\Nef(X)$, i.e., $D$ is the linear combination of classes of codimension $1$ irreducible subvarieties of $X$ with integer coefficients. Note also that here $c_2(X) \neq 0$. 
Indeed, otherwise there would exist a finite \'etale cover $A \to X$,
where $A$ is an abelian variety. Since we are assuming that $X$ 
is simply connected, this cannot happen.

\begin{remark}[{Concerning the hypothesis in Theorem \ref{thm:MS-CY3} (\ref{item:c2=0_semiample})}]
\label{rmk:when-CY-hypotheses-hold}
If the hypothesis in Theorem \ref{thm:MS-CY3} (\ref{item:c2>0_semiample}) fails, then 
as mentioned above, Miyaoka's theorem 
implies $Z:=c_2(X)^\perp\cap\Nef(X)$ is a non-zero face of $\Nef(X)$. A priori, $Z$ could be irrational. If $Z$ contains a non-zero rational class $D$, then the semi-ampleness conjecture \cite[Conjecture 2.1]{CY-semiample} implies that some scalar multiple $mD$ is a semi-ample divisor, 
and so the hypothesis in (\ref{item:c2=0_semiample}) holds. 

Thus, assuming the semi-ampleness conjecture, the only Calabi-Yau varieties $X$ that Theorem \ref{thm:MS-CY3} does not apply to are those for which $Z$ is non-zero and contains no non-zero rational classes. If \cite[Question-Conjecture 2.6]{semi-ampleness-conj} of Oguiso is true over $\bbar\QQ$, then this situation never occurs when the Picard number $\rho(X)$ is sufficiently large.
\end{remark}

In light of Remark \ref{rmk:when-CY-hypotheses-hold}, we have the following result.

\begin{corollary}
\label{cor:MS-for-all-3folds}
If the semi-ampleness conjecture \cite[Conjecture 2.1]{CY-semiample} and \cite[Question-Conjecture 2.6]{semi-ampleness-conj} are true over $\bbar\QQ$, then the Medvedev-Scanlon Conjecture \ref{conj:MS} is true for all automorphisms of smooth minimal threefolds of non-negative Kodaira dimension and sufficiently large Picard number.
\end{corollary}

\smallskip
\noindent\textbf{Acknowledgments.} 
We thank our colleagues Ekaterina Amerik, Donu Arapura, St\'ephane Druel, 
Najmuddin Fakhruddin, Jesse Kass, Brian Lehman, 
S\'andor Kov\'acs, Tom Scanlon, Alan Thompson, Burt Totaro, 
Tom Tucker, Junyi Xie, and Yi Zhu for stimulating conversations.

\section{The case of positive Kodaira dimension: proof 
of Theorem \ref{thm:MS-Kod>0}} 
\label{sec:shorter-pf}

We begin with two useful lemmas.

\begin{lemma}
\label{equivalent conj:MS}
In order to prove Conjecture~\ref{conj:MS} for 
the dynamical system $(X,\phi)$, it  is sufficient 
to prove Conjecture~\ref{conj:MS} for 
an iterate $(X,\phi^m)$, for some $m\in\NN$.
\end{lemma}

\begin{proof}
It is clear that if $\phi^m$ has a Zariski dense orbit, then 
so does $\phi$. 

It remains to show that if $\phi$ does not preserve a nonconstant 
fibration, then neither does $\phi^m$. Indeed, suppose there 
exists a nonconstant $f\in F(X)$ such that $(\phi^m)^*(f)=f$. 
Then $\phi$ preserves the symmetric function $g_i$ in the rational 
functions $f,\phi^*(f),\dots, (\phi^{m-1})^*(f)$, for 
each $i=1,\dots, m$.  Since $f$ is nonconstant, then at least one
of $g_1, \dots, g_m$ is non-constant. In other words,  
there exists a non-constant function $g_i$ which is nonconstant 
and thus fixed by $\phi^*$, as desired.
\end{proof}

\begin{lemma}
\label{l:fib-over-C}
Let $\phi:X\dasharrow X$ be a birational automorphism defined over a field
$k$. Let $F$ be an uncountable algebraically closed field containing $k$. 
Then the following conditions are equivalent:

\smallskip
(1) $k(X)^{\phi} = k$,

\smallskip
(2) There exists a $F$-point $x \in X(F)$ such that
the orbit $\{ \phi^n(x) \, | \, n = 0, 1, 2, \dots \}$ is dense in $X$. 

\smallskip
(3) $F(X)^{\phi} = F$.
\end{lemma}

\begin{proof} The implication (1) $\Longrightarrow$ (2) 
follows from \cite[Theorem 1.2]{dynamical-Rosenlicht}.

The remaining implications (2) $\Longrightarrow$ (3) 
and (3) $\Longrightarrow$ (1) are obvious.
\end{proof}


\begin{proof}[Proof of Theorem \ref{thm:MS-Kod>0}] 
Let $k$ be a finitely generated field such that both $X$ and $\phi$ are defined over $k$. Let $F$ be an algebraically closed field containing $k$; since $k$ is finitely generated, we may view $F$ as a subfield of $\mathbb{C}$.  
The theorem asserts that $F(X)^{\phi} \neq F$. 
By Lemma \ref{l:fib-over-C}, we may also assume that $F = \mathbb C$ 
is the field of complex numbers. 

Note that we may replace $X$ by a birationally equivalent variety;
this does not change $\CC(X)$ or $\CC(X)^{\phi}$.  After resolving the
singularities of $X$, we may also assume that $X$ is smooth.


Next, consider the Iitaka fibration, i.e.~the rational map $f:X\dasharrow\PP^N$ defined 
by the complete linear system $|mK_X|$ for $m$ sufficiently divisible. 
The image is a projective variety $Y$ of dimension $\kappa(X)$. 
Since $\phi^*K_X^{\otimes m}\simeq K_X^{\otimes m}$, we have 
an induced action $\bbar\phi$ on $\PP^N$ such that 
$f\circ\phi=\bbar\phi \circ f$.
By a theorem of Deligne and Ueno \cite[Thm 14.10]{ueno}, the image 
of the $m$-th pluricanonical representation 
$\rho_m:\bir(X)\to\GL(H^0(mK_X))$ is a finite group. 
Thus, after replacing $\phi$ by an iterate, we may assume that 
$\rho_m(\phi)=\id$; that is, $\bbar\phi$ is the identity on $\PP^N$.
So, $\phi$ preserves a rational fibration, as claimed. 
\end{proof}

\section{The Beauville-Bogomolov decomposition theorem over $\QQbar$}

We now recall the Beauville-Bogomolov decomposition theorem. 
Suppose $X$ is a smooth complex projective variety with numerically canonical
divisor $K_X$. 
Beauville \cite[p.~9]{BB-decomp} defines 
$\pi:\widetilde{X}\to X$ to be a \emph{minimal split cover} if it 
is a finite \'etale Galois cover, 
$\widetilde{X}\simeq A\times S$, where $A$ is an abelian 
variety and $S$ is simply connected, and there is no non-trivial 
element of the Galois group that simultaneously acts as translation 
on $A$ and the identity on $S$. The main theorem together with 
Proposition 3 of \cite{BB-decomp} show that every such $X$ has 
a minimal split covering and that it is unique up to non-unique isomorphism.

In the sequel we will need a variant of the Beauville-Bogomolov decomposition 
theorem \cite{BB-decomp} over $\QQbar$. For lack of a suitable reference, 
we will prove it below.

\begin{proposition}
\label{prop:BB-decomp-over-Qbar}
Let $X$ be a smooth projective minimal variety over $\QQbar$ with $K_X$ numerically trivial. Then there exists a finite \'etale Galois cover $\widetilde{X}\to X$ defined over $\QQbar$ such that
\begin{enumerate}
\item $\widetilde{X}=A\times \prod_i Y_i\times \prod_j Z_j$, where $A$ is an abelian variety, the $Y_i$ are Calabi-Yau, and the $Z_j$ are hyperk\"ahler,
\item no element of the Galois group acts simultaneously as translation on $A$ and the identity on all of the $Y_i$ and $Z_j$.
\end{enumerate}
\end{proposition}
\begin{proof}
The Beauville-Bogomolov decomposition theorem tells us that there is 
a finite group $G$ and 
a $G$-torsor $T\to X_\CC$ with $T=A\times\prod_i Y_i\times\prod_j Z_j$, where $A$ is an abelian variety, the $Y_i$ are Calabi-Yau varieties, and the $Z_j$ are hyperk\"ahler varieties. By a standard limit argument, there exists a finitely generated field extension $F/\QQbar$ so that we can descend $T\to X_\CC$ to a $G$-torsor $T'\to X_F$, the abelian variety $A$ to an abelian variety $A'$ over $F$, and the $Y_i$ (resp.~$Z_j$) to smooth proper $F$-schemes $Y'_i$ (resp.~$Z'_j$). Moreover, after possibly enlarging $F$, we can descend the isomorphism $T\simeq A\times\prod_i Y_i\times\prod_j Z_j$ to an isomorphism $T'\simeq A'\times\prod_i Y'_i\times\prod_j Z'_j$. Since $H^0(\Omega^p_{Y'_i})\otimes_F\CC=H^0(\Omega^p_{Y_i})$, we have $H^0(\Omega^p_{Y'_i})=0$ for $0<p<\dim Y'_i$. By similar reasoning, we see $H^0(\Omega^2_{Z'_j})$ is 1-dimensional and that $K_{Y'_i}\simeq\O_{Y'_i}$; the latter statement can be proved by using the fact that a line bundle $\L$ on a projective variety is trivial if and only if $H^0(\L)$ and $H^0(\L^\vee)$ are both non-zero. Choosing a generator $\omega_j\in H^0(\Omega^2_{Z'_j})$, we have an induced map $T_{Z'_j}\to\Omega^1_{Z'_j}$ and non-degeneracy of $\omega_j$ is equivalent to this map being an isomorphism.
 Since this is true after a field extension from $F$ to $\CC$, 
it is true over $F$.

Next, let $V$ be a smooth $\QQbar$-variety with function field $F$. After possibly shrinking $V$, we can extend $T'\to X_F$ to a $G$-torsor $T''\to X_V$, extend $A'$ to an abelian scheme $A''\to V$, $Y'_i$ and $Z'_j$ to smooth proper $V$-schemes $Y''_i$ and $Z''_j$, and can assume $T''\simeq A''\times\prod_iY''_i\times\prod_jZ''_j$ over $V$. Let $\pi_i:Y''_i\to V$ and $\psi_j:Z''_j\to V$ be the structure maps. After suitably shrinking $V$, we can assume $(\pi_i)_*\Omega^p_{Y''_i/V}=0$ for $0<p<\dim Y''_i$, that $(\psi_j)_*\Omega^2_{Z''_j/V}\simeq\O_V$, and that there is a non-vanishing section $\omega_j$ of $(\psi_j)_*\Omega^2_{Z''_j/V}$ whose induced map $T_{Z''_j/V}\to\Omega^1_{Z''_j/V}$ is an isomorphism.

Finally, we show that for all maps $t:\spec\CC\to V$, the complex analytifications of the $(Y''_i)_t$ and $(Z''_j)_t$ are simply connected. First note that by the Beauville-Bogomolov decomposition theorem, these varieties have virtually abelian fundamental groups; specifically, if $W$ denotes one of these varieties, then there is a finite Galois cover $A\times S\to W$ with $A$ an abelian variety and $S$ simply connected, so $\pi_1(W)$ contains $\pi_1(A)\simeq\ZZ^r$ as a finite index subgroup. Next, note that if the \'etale fundamental group $\pi_1^{et}(W)$ is trivial, then so is $\pi_1(W)$. Indeed, if $\pi_1^{et}(W)=0$, then $r=0$, so $\pi_1(W)$ is finite and therefore, $\pi_1(W)=\pi_1^{et}(W)=0$. Thus, it suffices to prove that for every geometric point $\bbar{v}$ of $V$, the \'etale fundamental groups $\pi_1^{et}((Y''_i)_{\bbar{v}})$ and $\pi_1^{et}((Z''_j)_{\bbar{v}})$ are trivial. Since the \'etale fundamental groups of the geometric generic fibers $(Y''_i)_{\bbar\eta}=Y_i$ and $(Z''_j)_{\bbar\eta}=Z_j$ are trivial, this follows immediately from specialization results of the \'etale fundamental group \cite[Proposition 0C0Q]{stacks-project}.

Now, choosing any $\QQbar$-point $v\in V$ gives our desired $G$-torsor of $T''_v\to X$.
\end{proof}

\section{Proof of Theorem \ref{thm:MS-reduction-nfolds}}
\label{sec:Kod-dim0--->CY-HK-Ab}

In this section we will prove Theorem \ref{thm:MS-reduction-nfolds}.
The key ingredients of the proof are supplied by Lemma~\ref{l:MS-finite-maps} 
and Proposition~\ref{prop:BBD-lifting} below.

\begin{lemma}
\label{l:MS-finite-maps}
Consider the commutative diagram
\[
\xymatrix{
X \ar@{-->}[r]^-{\phi} \ar[d]_{\pi} & X \ar[d]^{\pi} \\
Y \ar@{-->}[r]^-{\psi} & Y, }
\]
where $\pi \colon X \to Y$ is a dominant morphism of irreducible varieties, $\phi$ and $\psi$ are birational isomorphisms of $X$ and $Y$, respectively, and the entire diagram is defined over $\bbar \QQ$. Further suppose that $\dim(X) = \dim(Y)$ and $\bbar \QQ(X)^{\phi} = \bbar \QQ$. Then

\smallskip
(a) $\bbar \QQ(Y)^{\psi} = \bbar \QQ$.

\smallskip
In parts (b) and (c), assume further that $\pi \colon X \to Y$ is a $G$-torsor
for some finite smooth group scheme $G$.

\smallskip
(b) If $\phi$ is regular at $x \in X$,
then $\psi$ is regular at $y:= \pi(x) \in Y$.

\smallskip
(c) If the Medvedev-Scanlon conjecture holds for $X$, 
then there exists a point $y \in Y(\bbar \QQ)$ 
whose $\psi$-orbit is dense in $Y$.
\end{lemma}

\begin{proof} (a) follows from 
Lemma~\ref{l:fib-over-C} using the embedding of $\mathbb{C}(Y)$ into $\mathbb{C}(X)$ through the induced map $\pi^*$.

(b) The composition $\pi \circ \phi \colon X \dasharrow Y$ 
is a $G$-invariant rational map which is regular at $x$.
Hence, it descends to a rational map $Y \dasharrow Y$ which is
regular at $y$. Clearly, this map coincides with $\psi$.
In other words, $\psi$ is regular at $y$, as claimed.

(c) Since the Medvedev-Scanlon Conjecture holds
for $\phi$, there exists a point $x \in X(\bbar\QQ)$ 
such that the $\phi$-orbit of $x$ is dense in $X$. Using part~(b) for each iterate of $\phi$, we conclude that for each $n\in\mathbb{N}$ such that $\phi^n$ is defined at $x$, we have that $\psi^n$ is defined at $y := \pi(x)$. Furthermore, since the orbit of $x$ under $\phi$ is dense in $X$, we conclude that also the orbit of $y$ under $\psi$  
is dense in $Y$.
\end{proof}

\begin{remark}
\label{rmk:pseudo-aut}
Let $X$ be a minimal threefold with $K_X$ torsion. If $\phi$ is 
a birational automorphism of $X$, then as Lazi\'c shows 
in \cite[p.~197]{beyond-K} between Remarks 6.1 and 6.2, $\phi$ 
is a pseudo-automorphism, i.e.~neither $\phi$ nor $\phi^{-1}$ 
contracts a divisor. 
\end{remark}

\begin{proposition}
\label{prop:BBD-lifting}
Let $X$ be a smooth projective minimal variety over $\QQbar$ with $K_X$ numerically trivial, and let $\pi:\widetilde{X}\to X$ be a minimal split cover provided by Proposition \ref{prop:BB-decomp-over-Qbar}. Then for every birational automorphism $\phi$ of $X$ over $\QQbar$, there exists a birational automorphism $\widetilde{\phi}$ of $\widetilde{X}$ over $\QQbar$ such that $\pi\circ\widetilde{\phi}=\phi\circ\pi$.
\end{proposition}
\begin{proof}
We know that $\pi\colon\widetilde{X}\to X$ is a $G$-torsor for some finite \'etale group scheme $G$, and that $\widetilde{X}=A\times S$ with $S$ simply connected and $A$ an abelian variety. Since $X$ is smooth, $\phi$ is regular on an open subset $U\subseteq X$ with $X\setminus U$ having codimension at least 2. Consider the Cartesian diagram
\[
\xymatrix{
\widetilde{X}\times_XU\ar[r]\ar[d] & \widetilde{X}\ar[d]^\pi\\
U\ar[r]^-{\phi|_U} & X
}
\]
Since $X\setminus U$ has codimension at least 2, by \cite[Proposition 3.2]{martin-G-torsors}, the $G$-torsor $\widetilde{X}\times_XU\to U$ extends uniquely to a $G$-torsor $\pi':\widetilde{X}'\to X$. We therefore have a commutative diagram
\[
\xymatrix{
\widetilde{X}'\ar@{-->}[r]^-{\widetilde{\phi}}\ar[d]^{\pi'} & \widetilde{X}\ar[d]^\pi\\
X\ar@{-->}[r]^-{\phi} & X
}
\]
So, to finish the proof, it suffices to show that $\widetilde{X}'$ is split, i.e.~the product of an abelian variety and a simply connected variety; indeed, if $\widetilde{X}'$ is split, then by \cite[Proposition 3]{BB-decomp}, there exists a map $\alpha\colon\widetilde{X}'\to\widetilde{X}$ such that $\pi'=\pi\circ\alpha$. Then by degree considerations, $\alpha$ must be an isomorphism and so $\widetilde{\phi}\circ\alpha$ is the birational map whose existence we asserted in the statement of the proposition.

The rest of the proof is devoted to showing that $\widetilde{X}'$ is split. Since $\pi'$ is \'etale, we see $(\pi')^*K_X=K_{\widetilde{X}'}$ and so $K_{\widetilde{X}'}$ is numerically trivial. Thus, by Proposition \ref{prop:BB-decomp-over-Qbar}, there is a minimal split cover $p'\colon Y'\to\widetilde{X}'$ defined over $\QQbar$. After replacing $p'$ by a further \'etale cover, we can assume that $\pi\circ p'$ is Galois with group $\Gamma$, although now $p'$ is merely a split covering instead of a minimal split covering. Since $Y'$ is split, we know $Y'=B'\times T'$ with $T'$ simply connected and $B'$ an abelian variety. Let $H$ be the Galois group of $Y'$ over $\widetilde{X}'$.

Mimicking the argument in the first paragraph of the proof, we obtain a diagram
\[
\xymatrix{
Y'\ar@{-->}[r]^-{\psi}\ar[d]^{p'} & Y\ar[d]^p\\
\widetilde{X}'\ar@{-->}[r]^-{\widetilde{\phi}}\ar[d]^{\pi'} & \widetilde{X}\ar[d]^\pi\\
X\ar@{-->}[r]^-{\phi} & X
}
\]
where $p$ is an $H$-torsor and $\pi\circ p$ is a $\Gamma$-torsor. Indeed, by Remark \ref{rmk:pseudo-aut}, there exist open subsets $U$ and $V$ of $X$ whose complements have codimension at least 2 and such that $\phi|_U\colon U\to V$ is an isomorphism. Pulling back $Y'|_U\to\widetilde{X}'|_U$ via the isomorphism $\widetilde{X}'\setminus \widetilde{X}'|_U$, we obtain a Cartesian diagram
\[
\xymatrix{
Y' \ar@{}[r]|-*[@]{\supseteq}\ar[d]^{p'} & Y'|_U\ar[d] \ar[r]^{\simeq} & Z\ar[d] &\\
\widetilde{X}' \ar@{}[r]|-*[@]{\supseteq}\ar[d]^{\pi'} & \widetilde{X}'|_U\ar[d] \ar[r]^{\simeq} & \widetilde{X}|_V\ar[d]\ar@{}[r]|-*[@]{\subseteq} & \widetilde{X}\ar[d]^{\pi}\\
X \ar@{}[r]|-*[@]{\supseteq} & U\ar[r]^{\simeq} & V\ar@{}[r]|-*[@]{\subseteq} & X
}
\]
where the map $Z\to\widetilde{X}|_V$ is an $H$-torsor. Since $\pi$ is \'etale, hence codimension preserving, $\widetilde{X}\setminus \widetilde{X}|_V$ has codimension at least 2. By \cite[Proposition 3.2]{martin-G-torsors}, the $H$-torsor $Z\to\widetilde{X}|_V$ extends uniquely to an $H$-torsor $p\colon Y\to\widetilde{X}$. Since $X\setminus U$ has codimension at least 2, another application of \cite[Proposition 3.2]{martin-G-torsors} shows that $\pi\circ p$ is a $\Gamma$-torsor.

Since $Y$ is a finite \'etale cover of $\widetilde{X}=A\times S$, we necessarily have $Y=B\times T$ with $B$ an abelian variety and $T$ simply connected. Moreover, since $\widetilde{X}$ is the minimal split covering of $X$, the proof of \cite[Proposition 3]{BB-decomp} tells us that the $H$-action on $Y$ realizes $H$ as the normal subgroup of elements in $\Gamma$ acting simultaneously as translation on $B$ and the identity on $T$. As a result, $\widetilde{X}=(B/H)\times T$, so $p$ induces isomorphisms $T\simeq S$ and $B/H\simeq A$.

To finish the proof, it suffices to show that $H$ acts on $Y'$ as translation on $B'$ and the identity on $T'$. Indeed, provided we can show this, we then know that $\widetilde{X}'=(B'/H)\times T'$, hence it is split as desired. To prove that $H$ acts on $Y'$ as stated, we compare it with the $H$-action on $Y=B\times T$, which we already know acts as translation on $B$ and the identity on $T$. Since $\psi$ is an $H$-equivariant map by construction, it induces an $H$-equivariant birational map $\bbar\psi:B'\dasharrow B$ on Albanese varieties. Every rational map of abelian varieties is regular, so $\bbar\psi$ is in fact an isomorphism. Moreover, after suitable choice of origin, it respects the group structure. Given $\gamma\in H$, we know it acts on $B$ as translation $t_z$ by some $z$, so $\gamma$ acts on $B$ as $\bbar\psi^{-1}t_z\bbar\psi$ which is translation by $\bbar\psi^{-1}(z)$. Now, choosing a general point $b\in B'$, $\psi$ induces a birational map on fibers $T'=Y'_b\dasharrow Y_{\psi(b)}=T$ that commutes with the $H$-action. Since each $\gamma\in H$ acts as the identity on $T$, we see that the automorphism $\gamma\colon T'\to T$ agrees with the identity map on a dense open. As a result, it is the identity map, which proves our desired claim.
\end{proof}

\begin{proof}[{Proof of Theorem \ref{thm:MS-reduction-nfolds}}]
Let $X$ be a smooth projective minimal variety over $\bbar\QQ$ 
of Kodaira dimension 0, and let $\phi$ be a birational automorphism 
of $X$. The Abundance Conjecture \ref{conj:abundance} 
tells us that $K_X$ is numerically trivial. Then by 
Proposition~\ref{prop:BB-decomp-over-Qbar}, there exists a minimal 
split cover $\pi:\widetilde{X}\to X$ defined over $\bbar\QQ$. 
By Proposition \ref{prop:BBD-lifting}, $\phi$ lifts to a birational 
automorphism $\widetilde{\phi}$ of $\widetilde{X}$. 
By Lemma~\ref{l:MS-finite-maps}~(c), it is then enough 
to show that Medvedev-Scanlon holds for $\widetilde{\phi}$.
\end{proof}

\section{Proof of Theorem \ref{thm:MS-reduction-3folds}}
\label{subsec:Kod-dim0-3fold--->CY}

Our proof will rely on the following lemma.

\begin{lemma}
\label{lem:Cantat}
Consider the commutative diagram
\[
\xymatrix{
X \ar@{-->}[r]^-{\phi} \ar[d]_{\pi} & X \ar[d]^{\pi} \\
Y \ar@{->}[r]^-{\psi} & Y, }
\]
where $\pi \colon X \to Y$ is a dominant morphism of irreducible varieties, 
$\phi$ is birational isomorphisms of $X$, $\psi$ is an automorphism
of $Y$, and the entire diagram is defined over $\bbar \QQ$. Suppose
$\bbar \QQ(X)^{\phi} = \bbar \QQ$ (and hence, 
$\bbar \QQ(Y)^{\phi} = \bbar \QQ$; see Lemma~\ref{l:MS-finite-maps}(a)),
and there exists a $y \in Y(\bbar \QQ)$ whose $\psi$-orbit is dense in $Y$.
Assume further that either (a) $\pi$ is birational 
or (b) $\phi$ is a (regular) automorphism and
$\dim(X) = \dim(Y) + 1$.
Then there exists an $x \in X(\bbar \QQ)$ whose $\phi$-orbit is dense in $X$.
\end{lemma}

\begin{proof}
(a) Suppose $\pi$ restricts to an isomorphism between dense open subsets $X_0$
of $X$ and $Y_0$ of $Y$. After replacing $y$ by an iterate, we may 
assume that $y \in Y_0$. We claim that the preimage $x \in X_0$ of $y$
has a dense $\phi$-orbit in $X$. Indeed, set $y_n := \psi^n(y) \in Y$. 
Then there is a sequence $i_1 \leqslant i_2 \leqslant \dots$ such 
that the points $y_{i_1}, y_{i_2}, \dots, $ all lie in $Y_0$ are
are dense in $Y$. Then $x_n : = \phi^n(x)$ are well defined for 
$n = i_1, i_2, \dots$ and are dense in $X$.  This proves the claim.

(b) By \cite[Theorem 1.2]{Dixmier}, 
$X$ has only finitely many $\phi$-invariant codimension $1$
subvarieties.  Denote their union by $H \subset X$. Once again,
set $y_n := \psi^n(y) \in Y$.  The union of the fibers $\pi^{-1}(y_n)$,
as $n$ ranges over the non-negative integers, is dense in $X$. Hence, 
one of these fibers is not contained in $H$. After replacing
$y$ by an iterate, we may assume that $\pi^{-1}(y) \not \subset H$. 
Choose a $\bbar \QQ$-point $x \in \pi^{-1}(y)$ which does not lie in $H$. 
We claim that the $\phi$-orbit of $x$ is dense in $X$. Indeed,
denote Zariski closure of the orbit of $x$ by $Z$. By our construction
$\pi(Z)$ contains the $\psi$-orbit of $y$ and thus is dense in $Y$.
Hence, $\dim(Y) \leqslant \dim(Z) \leqslant \dim(X) = \dim(Y) + 1$.
On the other hand, since $x \not \in H$, $Z$ cannot be a hypersurface in $Y$.
Thus $\dim(Z) = \dim(X) = \dim(Y) + 1$, i.e., $Z = X$, as desired.
\end{proof}

We now proceed with the proof of Theorem \ref{thm:MS-reduction-3folds}.
Since the abundance conjecture is known for threefolds \cite{abundance-3folds}, we can apply Theorem \ref{thm:MS-reduction-nfolds}. Thus, 
the Medvedev-Scanlon Conjecture \ref{conj:MS} holds for all smooth 
projective minimal threefolds of Kodaira dimension 0 if and only if 
it holds for products of Calabi-Yau varieties, hyperk\"ahler varieties, 
and abelian varieties over $\bbar\QQ$. We are therefore reduced to three 
possibilities: (i) $X$ is an abelian threefold, (ii) $X$ is a product
$E\times S$, where $E$ is an elliptic curve and $S$ is a K3 surface, 
or (iii) $X$ is a smooth Calabi-Yau $3$-fold. 
The Medvedev-Scanlon conjecture holds in case (i) by \cite{MS-ab-var}.
The main result of this section,
Proposition~\ref{prop:ExS-prod-birat-aut}, asserts that 
Conjecture~\ref{conj:MS} also holds in case (ii). This will leave us 
with case (iii), thus completing the proof of 
Theorem \ref{thm:MS-reduction-3folds}.

\begin{lemma} \label{lem.ExS}
Suppose $X=E\times S$, where $E$ an elliptic curve and $S$ is 
a smooth minimal surface with trivial Albanese and $\kappa(S)\geq 0$.  
Every birational isomorphism $\phi \colon X \dasharrow X$
is of the form  $\phi=\phi_E\times\phi_S$ with $\phi_E$ an automorphism 
of $E$ and $\phi_S$ an automorphism of $S$. In particular, 
every birational isomorphism of $X$ is regular.
\end{lemma}

\begin{proof}
The projection $\pi:X\to E$ is the Albanese map for $X$. Thus
$\phi$ induces a birational automorphism $\phi_E$ of $E$ such 
that $\pi\circ\phi=\phi_E\circ\pi$. 
Since $E$ is a smooth curve,
$\phi_E$ is an automorphism of $E$. Replacing $\phi$ by 
$\phi\circ(\phi_E^{-1},\id_S)$, we see that to prove 
the lemma, we may assume $\phi_E=\id_E$.

Since $X$ is smooth, the indeterminacy locus $I(\phi)$ of $\phi$ 
has codimension at least $2$, and so $I(\phi) \cap X_t$ has codimension 
at least $1$ for all $t\in E$. We therefore obtain a map 
$f \colon E\to \bir(S)$ given by $t \mapsto \phi|_{X_t}$. Since 
$\kappa(S) \geq 0$, $S$ is not ruled, so $S$ is a unique smooth minimal
surface in its birational class, and 
$\bir(S) = \aut(S)$, see for example \cite[Theorem V.19]{beauville-surfaces}. 
Our goal is to show that the resulting map $f \colon E \to \aut(S)$ 
is constant. Choose a point $t_0 \in E$ and let $\sigma := f(t_0) \in \aut(S)$.
After composing $\phi$ with 
$(1, \sigma^{-1}) \colon E \times S \to E \times S$, we may assume
that $f(t_0) = 1 \in \aut(S)$. Since $E$ is irreducible, this implies that
the image of $f$ lies in $\aut^0(S)$.  Since $S$ has trivial Albanese, 
by \cite[Corollary 5.8]{fujiki-lieberman}, $\aut^0(S)$ is an affine 
algebraic group. Thus, $f$ must be a constant map, as claimed.
We now define $\phi_S$ to be the image of this map.
\end{proof}

\begin{proposition}
\label{prop:ExS-prod-birat-aut}
Suppose $X=E\times S$, were $E$ an elliptic curve and $S$ is 
a surface with trivial Albanese and $\kappa(S)\geq 0$. 
Let $\phi \colon X \dasharrow X$ be a birational isomorphism such that
$\bbar \QQ(X)^{\phi} = \bbar \QQ$.
Then Conjecture~\ref{conj:MS} holds for $(X, \phi)$.
\end{proposition}

\begin{proof}
Let $\pi \colon S \to S_{min}$ be the minimal model of $S$.
By Lemma~\ref{lem.ExS}, $\phi$ descends to 
an automorphism $E \times S_{min} \to E \times S_{min}$
of the form $(\phi_E, \phi_{min})$, where $\phi_E$ is an automorphism of $E$
and $\phi_{min}$ is an automorphism of $S_{min}$. Now
consider the commutative diagram
\[
\xymatrix{
E \times S  \ar@{-->}[rr]^-{\phi} \ar[d]_{\id \times \pi} & & E \times S 
\ar[d]_{\id \times \pi} \\
E \times S_{min}  \ar@{->}[rr]^--{\phi_E \times \phi_{min}} 
\ar[d]_{\text{pr}} & & E \times S_{min} \ar[d]_{\text{pr}} \\
S_{min} \ar@{->}[rr]^-{\phi_{min}} & & S_{min}, } \]
By~\cite[Theorem~1.3]{BGT-2} the Medvedev-Scanlon conjecture 
holds for the automorphism
$\phi_{min}$ of the surface $S_{min}$. By Lemma~\ref{lem.ExS}(b), 
$E \times S_{min}$ has a $\bbar \QQ$-point with a dense 
$(\phi_E,  \phi_{min})$-orbit. Applying
Lemma~\ref{lem.ExS}(a), we conclude that the automorphism
$E \times S$ has a $\bbar \QQ$-point with a dense $\phi$-orbit, as desired.
\end{proof}

\section{Pseudo-automorphisms that preserve a line bundle}

The following result will be used in the proof of 
Theorem~\ref{thm:MS-CY3} in the next section. 

\begin{proposition} \label{prop.big}
Suppose $\phi \colon X \dasharrow X$ is pseudo-automorphism of a smooth
projective variety defined over a field $k$ of characteristic $0$,
$L$ is a line bundle such that $\phi^*(L) \simeq L$, and $Y$ is the closure of the image of the natural rational map 
$i \colon X \dasharrow \PP H^0(X, L)$
for large $n$. Here, as usual $H^0(X, L)$
denotes the finite-dimensional space of global sections of $L$, 
and $\PP H^0(X, L)$ is the associated projective space. Then

\smallskip
(a) $\phi$ induces a linear automorphism $\bar{\phi}$ of 
the projective space $\PP H^0(X, L)$ preserving $Y$.

\smallskip
\noindent
Moreover, assume $k(X)^{\phi} = k$. Then

\smallskip
(b) there is a dense $\bar{\phi}$-invariant subset $U$ of $Y$ such that
the $\bar{\phi}$-orbit of $y$ is dense in $Y$ for every $y \in U$,

\smallskip
(c) $Y$ is a rational variety over the algebraic closure $\overline{k}$.
\end{proposition}

Note that since $\phi$ is a pseudo-automorphism, it induces an
automorphism $\phi^* \colon \Pic(X) \to \Pic(X)$.

\begin{proof}
(a) We begin with the following preliminary observation.
Suppose $L$ and $L'$ are isomorphic line bundles on a complete variety
$X$ defined over $k$.  We claim that there is a canonically defined linear
isomorphism between the finite-dimensional projective spaces
$\PP H^0(X, L)$ and $\PP H^0(X, L')$. To define this linear isomorphism,
write $L = \mathcal{O}_X(D)$ and $L' = O_X(D')$, where
$D$ and $D'$ are divisors on $X$. Since $L$ and $L'$ are isomorphic,
these divisors are linearly equivalent. That is,
\begin{equation} \label{e.divisors}
D' = D + (f), 
\end{equation}
where $(f)$ denotes the divisor associated
to a rational function $f \in k(X)$. Once $f$ is chosen,
we can define an isomorphism of vector spaces
$H^0(X, L) \to H^0(X, L')$ given by $\alpha \mapsto f \alpha$. 
The rational function $f$ in~\eqref{e.divisors} 
is uniquely determined by $L$ and $L'$ up to 
a non-zero scalar factor.
The isomorphism of projective spaces $\PP H^0(X, L) \to \PP H^0(X, L')$
thus defined depends only on $L$ and $L'$ and  
not on the choice of $f$.  This proves the claim.

We now apply this claim in the setting of the proposition,
with $L' := \phi^*(L)$. The line bundles $L$ and $L'$  
are isomorphic by our assumption. On the other hand,
$\phi$ induces an isomorphism
\[ \phi^*: H^0(X, L)  \to H^0(X, L') \] 
via pull-back. Composing with the inverse of the linear isomorphism
$\PP H^0(X, L) \to \PP H^0(X, L')$ constructed above,
we obtain a desired automorphism 
$\bar{\phi} \colon \PP H^0(X, L) \to \PP H^0(X, L)$ such that the diagram   
\[ \xymatrix{ X \ar@{-->}[r]^-\phi \ar@{-->}[d]_i & X \ar@{-->}[d]^i \\
    \PP H^0(X, L) \ar[r]^-{\bar{\phi}} & \PP H^0(X, L) }
\]
commutes.

(b) Let $Y$ be the closure of image of $X$ in $\PP(V)$ under 
$i$, where $V : = H^0(X, L)$. Since $k(X)^{\phi} = k$,
clearly $k(Y)^{\phi} = k$ as well. 

Set $G$ to be the subgroup
of $\PGL(V)$ consisting of automorphisms of $\PP(V)$ which preserve
$Y$. Then $\bar{\phi} \in G$, and $G$ is a closed subgroup of 
$\PGL(V)$ and hence, a linear algebraic group. Let $G_0$ be 
the  Zariski closure of the subgroup generated 
by $\bar{\phi}$ inside $G$. Then $G_0$ is an abelian linear algebraic group.
Moreover, for any $y \in Y$, the orbit of $y$ under $\phi$ has the same
closure in $Y$ as the orbit of $y$ under $G_0$. So, it suffices to show that 
there is a dense open subset $U \subset Y$ such that every $y \in U$ has
a dense orbit under $G_0$. The last assertion is a consequence of
Rosenlicht's theorem; see~\cite[Theorem 2]{rosenlicht}, \emph{cf}.~also
\cite[Theorem 1.1]{dynamical-Rosenlicht} 
and~\cite[Proposition 7.4(1)]{Dixmier}; in fact, we can take $U$ to 
be a dense $G_0$-orbit in $Y$.

(c) Since $U$ is a $G_0$-orbit, it is isomorphic to the homogeneous
space $G_0/H_0$, for some subgroup $H_0 \subset G_0$. Since $G_0$ 
is abelian, $H_0$ is normal in $G_0$. Hence, as a variety, $U$ is 
isomorphic to the abelian linear algebraic group $G_0/H_0$.
Every abelian linear irreducible algebraic group 
over $\overline{k}$ is isomorphic to a direct 
product of copies of $\mathbb G_a$ and $\GG_m$; we conclude that
$U$ is rational over $\overline{k}$ and hence, so is $Y$.
\end{proof}

\section{Proof of Theorem \ref{thm:MS-CY3}}
\label{sect.MS-CY3(1)}

Let $X$ be a minimal threefold with $K_X$ torsion. 
Then by Remark \ref{rmk:pseudo-aut}, $\phi$ 
is a pseudo-automorphism, i.e.~neither $\phi$ nor $\phi^{-1}$ 
contracts a divisor. As a result, $\phi$ induces an automorphism 
of the nef cone $\Nef(X)$. Every smooth minimal threefold $Y$ 
with $c_2(Y) = 0$ has an \'etale cover by an abelian variety, 
so if $X$ is a Calabi-Yau variety, hence simply connected, we must 
have $c_2(X) \neq 0$. 
As mentioned in the introduction, a theorem of Miyaoka \cite{Miyaoka} 
then tells us that $c_2(X)$ is positive on the ample cone 
$\Amp(X)$ and non-negative on $\Nef(X)$. We first consider 
the case where $c_2(X)$ is strictly positive on the nef cone. 
This approach is based on arguments given in Chapter 4 
of~\cite{wild-aut-thesis}.

\begin{lemma}
\label{l:compact-cone}
Suppose $\ell: \RR^n \to \RR$ is a linear function and $C$ is 
a closed cone in $\RR^n$ such that $\ell(z) > 0$ for any $z \in C$ 
other than the origin. Then for any real number $M \geq 0$, 
the region $C_M := \{ z \in C \mid \ell(z) \leq M\}$ is compact.
\end{lemma}

\begin{proof}
Let $S$ be the intersection of $C$ with the unit sphere. Clearly $S$ is compact. Define the function $f: S \to \RR$ given as follows. For $p \in S$, let $I_p$ be the intersection of the line through $p$ and the origin with the strip $0 \leq \ell(z) \leq M$. Since $\ell$ is positive on $C$, $I_p$ is an interval of finite length. Let $f(p)$ be the length of $I_p$. Since $f$ is continuous and $S$ is compact, $f$ attains its maximal value $r$ on $S$. Consequently, $C_M$ is contained in the ball of radius $r$ centered at the origin. Thus $C_M$ is closed and bounded, hence compact.
\end{proof}


\begin{proof}[{Proof of Theorem \ref{thm:MS-CY3}}]
(1) Since $c_2(X)$ is strictly positive on $\Nef(X)$, Lemma \ref{l:compact-cone} 
shows that for all $M\geq0$, the region $\{D\in\Nef(X)\mid c_2(X)\cdot 
D\leq M\}$ is compact. As a result, $c_2(X)$ achieves a minimum positive 
value on $\Pic(X) \cap \Amp(X)$ and this value is achieved 
by only finitely many $D_i$. Taking the sum of these finitely 
many $D_i$, we obtain an ample class $A$ which is fixed 
by $\phi^*$. Let $\M$ be an ample line bundle representing 
the class of $A$. Since the Albanese of $X$ is trivial, rational 
equivalence is the same as linear equivalence. Since 
$\phi^*A = A$ in $\NS(X)\otimes\CC$, we have 
$\phi^*\M \simeq \M\otimes\N$ where $\N$ is a torsion line bundle. 
Replacing $A$ by a scalar multiple, we may assume that $\phi^*(A)$ 
is isomorphic to $A$ and that $A$ is very ample. If $\phi$ preserves a rational fibration, we are done. Otherwise, with notation as 
in Proposition~\ref{prop.big}(b), there is a dense set of 
$y\in Y$ with dense orbit under $\bar\phi$. However, $A$ is 
very ample, so $Y=X$ which gives the desired conclusion.

(2) We will now consider the case where there is a semi-ample 
divisor $D \neq 0$ on $X$ such that $c_2(X)\cdot D=0$. 
Let $\pi:X\to Y$ be the associated $c_2$-contraction. Oguiso 
shows (\cite[Theorem 4.3]{semi-ampleness-conj}) that there are only 
finitely many $c_2$-contractions, and so after replacing $\phi$ by 
a further iterate, we can assume $\phi^*[D] = [D]$. By 
Proposition~\ref{prop.big}(a),
$\phi$ descends to an automorphism $\bbar{\phi}$ of $Y$.
Since $D$ is non-zero, $Y$ is not a point. 
We now consider three cases. 

\smallskip
{\bf Case 1:} $\dim(Y) = 3$, i.e., $D$ is big. Since contractions have connected fibers, $\pi$ is birational. If $X$ preserves a rational fibration, we are done. Otherwise, Proposition~\ref{prop.big}(c) tells us that $Y$ is rational over $\QQbar$, which is not possible since $X$ has Kodaira dimension 0. So, the Medvedev-Scanlon Conjecture for $\phi$ holds in this case.

\smallskip
{\bf Case 2:} $\dim(Y) = 2$. Here the Medvedev-Scanlon conjecture 
holds by Lemmas~\ref{l:fib-over-C} and~\ref{lem:Cantat}.

\smallskip
{\bf Case 3:} $\dim(Y) = 1$.
By Proposition~\ref{prop.big}(c), $Y \simeq \mathbb P^1$ (over $\QQbar$). 

Let $Z\subseteq \PP^1$ be the locus of points $t$ where 
the fiber $X_t$ is singular. Then $\bbar\phi(Z)=Z$. Since $Z$ is a finite 
set, after replacing $\phi$ by a further iterate, we can assume 
$\bbar\phi$ fixes $Z$ point-wise. By \cite[Theorem 0.2]{mfds-over-curves}, 
we know that $Z$ contains at least 3 points. It follows that $\bbar\phi$ 
is the identity since it fixes at least three points of $\PP^1$. 
In other words, there exists a rational function on $X$ which 
is invariant under some iterate of $\phi$, a contradiction.
\end{proof}

\end{document}